\documentclass{article}
\usepackage{graphicx} 
\usepackage{amssymb}
\usepackage{amsmath}
\usepackage{amsthm}
\usepackage{bbm}
\usepackage{enumitem}
\usepackage{varwidth}
\usepackage{tasks}
\usepackage{float}
\usepackage{thmtools, thm-restate}
\usepackage{authblk}
\usepackage{hyperref}

\newtheorem{thm}{Theorem}

\newtheorem{lem}{Lemma}
\newtheorem{cor}{Corollary}

\theoremstyle{remark}
\newtheorem{rem}{Remark}

\newcommand{\Q}{\mathbb{Q}}
\newcommand{\Z}{\mathbb{Z}}
\newcommand{\mycomment}[1]{}
\newcommand{\email}[1]{\protect\href{mailto:#1}{#1}}

\title{Forbidden multipliers in abelian difference sets}
\author{Niklas Miller\thanks{Department of Mathematics and Systems Analysis, Aalto University, Finland, \email{niklas.miller@aalto.fi}. This work has been supported by the Research Council of Finland (Grant \#351271, PI~Camilla~Hollanti).}}
\date{\today}
\begin{document}
\maketitle
\begin{abstract}
    We make the observation that certain group automorphisms that fix a large subgroup of an abelian group cannot be multipliers in any non-trivial abelian difference sets, with the single exception of an involution that can be a multiplier in Hadamard difference sets, provided that the difference set contains a sub-difference set of the same type. We use this observation together with a multiplier theorem to rule out the existence of difference sets, and derive bounds for the numerical multiplier group of a difference set.
\end{abstract}
\smallskip
\noindent \qquad \ \textbf{\small{Key words.}} Difference set, multiplier group, multiplier theorem.

\smallskip
\noindent \qquad \ \textbf{\small{MSC Codes.}} 05B10.

\section{Introduction}

A $(v,k,\lambda)$ difference set $D$ is a $k$-subset of an abelian group $G$ of order $v$, such that every non-identity element of $G$ is represented as a difference of elements of $D$ in exactly $\lambda$ ways. A simple count gives the equation
\begin{align}
\label{eq:fundamental}
    k(k-1)=\lambda(v-1).
\end{align}
Let $n=k-\lambda$ be the order of $D$; we assume that $k<\frac{v}{2}$ and restrict attention to non-trivial difference sets with $n>1$. One has the inequalities 
\begin{align}\label{eq:basic}
4n-1\leq v\leq n^2+n+1,    
\end{align}
where the lower bound is attained for $(4n-1,2n-1,n-1)$ Paley--Hadamard difference sets, \emph{i.e.}, the set of non-zero quadratic residues in the additive group of a finite field of prime power order $q=4n-1$, and the upper bound is attained for $(n^2+n+1,n+1,1)$ finite projective planes, which are conjectured to exist only for $n$ a prime power. All groups considered in this article are abelian.

Difference sets correspond to symmetric designs with regular abelian automorphism groups and are thus of interest in combinatorial design theory. They have applications in signal design in digital communications, synchronization and radar.

The main result of this article is a theorem, which restricts the multiplier groups of abelian difference sets in groups satisfying certain conditions. Recall that Hadamard difference sets have parameters $(4u^2,2u^2-u,u^2-u)$ for some positive integer $u$. Their $\pm 1$ incidence matrices correspond to regular Hadamard matrices. We denote by $C_k$ the cyclic group of order $k$.

\begin{thm}
\label{thm:forbidden}
Let $G=C_{p^e}\times H$ where $e\geq 2$ and $H$ is an abelian group whose Sylow $p$-subgroup has exponent less than $p^e$. Suppose that $t$ is an integer such that $t\equiv 1+p^{e-1}\pmod{p^{e}}$ and $t\equiv 1\pmod{\text{exp}(H)}$. Then $t$ is not a multiplier in any non-trivial $(v,k,\lambda)$ difference set $D$ in $G$, with the possible exception $p=2$ and $D$ is a Hadamard difference set in a non-cyclic group.
\end{thm}

This theorem is an improvement on \cite[Theorem 2]{arasu2001nonexistence}, which under the same assumptions arrives at the conclusion that $t$ is not a multiplier if $4(p-1)n>v$. In view of (\ref{eq:basic}), their theorem does not cover the case $p=2$, which is covered by Theorem \ref{thm:forbidden}. In fact, when $p=2$, it is possible for $t$ to be a multiplier, but only in a Hadamard difference set that contains a sub-difference set of the same type in a certain subgroup (Theorem \ref{thm:interesting}).

\section{Mathematical preliminaries}

We will view a $(v,k,\lambda)$ difference set $D\subset G$ as the element $D=\sum_{g\in D} g$ in the group ring $\mathbb{Z}[G]$, so that the defining equation becomes
\begin{align}\label{eq:defining}
DD^{(-1)}=n+\lambda G.    
\end{align}
For a group ring element $F=\sum_{g\in G} f_g g\in\mathbb{Z}[G]$, we use the notation $F^{(t)}=\sum_{g\in G} f_g g^t\in\mathbb{Z}[G]$ for any integer $t$ relatively prime to $\text{exp}(G)$. We also let $|F|=\sum_{g\in G} f_g$. An integer $t$ is said to be a (numerical) multiplier of $D$ if $D^{(t)}=Dg$ for some $g\in G$; the group of multipliers is denoted by $M(D)$. We will use the fact that every difference set has a translate fixed by every multiplier \cite{mcfarland1978translates}.

If $H$ is a subgroup of $G$, then applying the canonical epimorhism $\rho:\Z[G]\to \Z[G/H]$ to equation (\ref{eq:defining}) shows that
\begin{align}\label{eq:}
    \rho(D)\rho(D)^{(-1)} = n+\lambda |H|(G/H).
\end{align} 

Let $\widehat{G}$ be the character group of $G$, that is, the group of homomorphisms $\chi:G\to\mathbb{C}^\times$, with product of characters defined by $(\chi_1\chi_2)(g)=\chi_1(g)\chi_2(g)$ for all $g\in G$. The characters extend to $\mathbb{Z}[G]$ in the natural way. For any subgroup $H$ of $G$, we let $H^\perp$ be the subgroup of characters that act trivially on $H$. There is the following character theoretic description of a difference set, which follows by applying characters to (\ref{eq:}) and recalling orthogonality relations.
\begin{align}\label{eq:characters}
\chi(\rho(D))\overline{\chi(\rho(D))} &= \begin{cases}
			n,& \text{if $\chi\in \widehat{G/H}$ is non-principal}\\
			k^2,& \text{if $\chi\in \widehat{G/H}$ is principal}.
		\end{cases}
\end{align} 
A group ring element $A\in\mathbb{Z}[G]$ is completely determined by its character values as follows.
\begin{lem}[Fourier inversion formula]\label{lem:Fourier_inversion}
    Let $A=\sum_{g\in G} a_g g \in \Z[G]$. Then $$a_g=\frac{1}{|G|}\sum_{\chi \in \widehat{G}}\chi(Ag^{-1}) \ \ \text{for all $g\in G$}. $$
\end{lem}
One ingredient needed to obtain Theorem \ref{thm:forbidden} is a bound on the coefficients of $\rho(D)$, which are sometimes called ``intersection numbers''. This will follow from (\ref{eq:characters}) and Lemma \ref{lem:Fourier_inversion}.

\begin{lem}
\label{lem:projection_coefficient_bound} Let $G$ be a finite abelian group, $H$ a subgroup of $G$ and $M=G/H$. Let $\rho:G\to M$ be the canonical epimorphism. If $D$ is a $(v,k,\lambda)$ difference set of order $n$ in $G$ and $\rho(D)=\sum_{m\in M} a_m   m\in\mathbb{Z}[M]$, then
    $$a_m\leq \frac{k+(|M|-1)\sqrt{n}}{|M|}$$
    for all $m\in M$.
\end{lem}

\begin{proof}
  Let $A=\rho(D)$. By (\ref{eq:characters}) and Lemma \ref{lem:Fourier_inversion},
  \begin{align*}
      |M|a_m=\left|\sum_{\chi\in\widehat{M}}\chi(Am^{-1})\right|\leq \sum_{\chi\in\widehat{M}}|\chi(A m^{-1})|=\sum_{\chi\in\widehat{M}}|\chi(A)|=k+(|M|-1)\sqrt{n}.
  \end{align*}
\end{proof}

We will prove a multiplier theorem in Section \ref{sec:mults}, which together with Theorem \ref{thm:forbidden} gives a non-existence criterion that can be used to rule out many open cases of abelian difference sets. For this purpose, we recall some terminology in algebraic number theory. A positive integer $m$ is called self-conjugate modulo $n$, if for every prime divisor $p\mid m$ there exists an integer $j_p$ such that $p^{j_p}\equiv -1\pmod{n'}$, where $n'$ is the $p$-free part of $n$. If $K/\mathbb{Q}$ is a Galois extension of number fields, then the Galois group $\text{Gal}(K/\mathbb{Q})$ acts transitively on the prime ideals $\mathfrak{p}_1,\dots,\mathfrak{p}_g$ lying above a prime $p$ in the ring of integers $\mathcal{O}_K$ in $K$. The decomposition group $D(\mathfrak{p}_i| p)$ is the stabilizer of $\mathfrak{p}_i$ under this action, which for abelian extensions depends only on $p$ and is denoted by $D(p)$. Here we are interested in decomposition groups in cyclotomic fields, since character values of difference sets are cyclotomic integers. We define $\zeta_n=\exp\left(\frac{2\pi i}{n}\right)$.

\begin{lem}[\cite{marcus1977number}]
    \label{lem:decomp}
    Let $n$ be a positive integer, let $K=\Q(\zeta_n)$ and let $p$ be a prime. Write $n=p^{a} n'$ where $\gcd(p,n')=1$. The Galois group of $K/\mathbb{Q}$ is $$\text{Gal}(K/\mathbb{Q})=\{\sigma_a:\zeta_n\mapsto\zeta_n^a:\operatorname{gcd}(a,n)=1\}\cong(\mathbb{Z}/n\mathbb{Z})^\times$$
    and the prime ideal decomposition of $(p)$ in $\mathcal{O}_K=\Z[\zeta_n]$ is given by
    $$p\mathcal{O}_K=(\mathfrak{p}_1\cdots \mathfrak{p}_g)^e$$
    where $e=\varphi(p^a)$, and $g=\frac{\varphi(n')}{\text{ord}_{n'}(p)}$. The decomposition group of $p$ is given by
    $$D(p)=\{\sigma_a\in\text{Gal}(K/\mathbb{Q}): a\equiv p^j\pmod{n'}\text{ for some $j\in\mathbb{Z}$}\}.$$
\end{lem}

\section{Forbidden multipliers and sub-difference sets}

We now give the proof of Theorem \ref{thm:forbidden}, after which we investigate more closely the prime $p=2$.

\begin{restatable}{theorem}{}\label{thm:forbidden}
Let $G=C_{p^e}\times H$ where $e\geq 2$ and $H$ is an abelian group whose Sylow $p$-subgroup has exponent less than $p^e$. Suppose that $t$ is an integer such that $t\equiv 1+p^{e-1}\pmod{p^{e}}$ and $t\equiv 1\pmod{\text{exp}(H)}$. Then $t$ is not a multiplier in any non-trivial $(v,k,\lambda)$ difference set $D$ in $G$, with the possible exception $p=2$ and $D$ is a Hadamard difference set in a non-cyclic group.
\end{restatable}

\begin{proof}
    The first part of the proof is the same as the proof of \cite[Theorem 2]{arasu2001nonexistence}. Write $G=\langle g\rangle \times H$ and let $D$ be a difference set in $G$ fixed by $t$. Let $K=\langle g^p\rangle \times H$ and let $P=\langle g^{p^{e-1}}\rangle\times \{1\}$.
    Then by considering the orbits of the action of $t$ on $G$ it is easy to see that $D=A+PB$ for some $A\subset K$ and $B\subset G\setminus K$. Now note that
    \begin{align}\label{eq:difs}
        DD^{(-1)}=AA^{(-1)}+PAB^{(-1)}+PA^{(-1)}B+pPBB^{(-1)}=n+\lambda G.
    \end{align}
    Comparing coefficients of non-identity elements of $P$ we conclude that $p|B|\leq \lambda$. On the other hand, $|A|+p|B|=k$ so $|A|\geq k-\lambda=n$. Let $\rho:\mathbb{Z}[G]\to \mathbb{Z}[G/K]$ be the canonical epimorphism. Then the coefficient of identity in $\rho(D)$ is $|A|$, so by Lemma \ref{lem:projection_coefficient_bound},
    $$n\leq |A|\leq\frac{k+(p-1)\sqrt{n}}{p}.$$
    Rearranging, we obtain $k-\sqrt{n}\leq \lambda\frac{p}{p-1}\leq 2\lambda$ and thus $0\leq n-\lambda\leq \sqrt{n}$. But then $n+\lambda(v-4n)=(n-\lambda)^2\leq n$, so that $v\leq 4n$. Inequality (\ref{eq:basic}) implies that $v=4n-1$ or $4n$. In the former case, $D$ is a $(4n-1,2n-1,n-1)$ Paley--Hadamard difference set and since $2\nmid 4n-1$, $p\geq 3$. But then $k-\sqrt{n}\leq \frac{3}{2}\lambda$ implies $n=1$, so $D$ is trivial contrary to assumption. If $D$ is a Hadamard difference set, then the inequality $k-\sqrt{n}\leq \lambda \frac{p}{p-1}$ implies $2\leq \frac{p}{p-1}$ so $p=2$, and this case is handled in Theorem \ref{thm:interesting}.
\end{proof}

McFarland \cite{mcfarland1990sub} proved that if $H$ and $K$ are abelian groups of relatively prime orders such that $|H|$ is even and $|K|$ is self-conjugate modulo $\text{exp}(H)$, then the existence of a Hadamard difference set in $H\times K$ implies the existence of a Hadamard difference set in $H$. More generally, he notes that if $D$ is a difference set in a group $H\times K$ and if the image of $D$ under $\rho: \mathbb{Z}[H\times K]\to \mathbb{Z}[H]$ is of the form $\rho(D)=aE+b(H-E)$ for some subset $E\subset H$ and integers $a$, $b$, then $E$ is a difference set in $H$ \cite{mcfarland1990sub}.

The sub-difference sets that we consider next are not of the above type. Instead, suppose that $D$ is a difference set in a group $G$, and suppose that for some subgroup $H$ of $G$, the image of $D$ under $\rho:\mathbb{Z}[G]\to\mathbb{Z}[G/H]$ is of the form
$$\rho(D)=aM+bhA,$$
for some $a,b\in\mathbb{Z}$ and $h\in G/H$, where $M$ is a subgroup of $G/H$ and $A\subset M$. Then it is easy to verify that $A$ is a difference set in $M$ of order $n/b^2$.

\begin{thm}
	\label{thm:interesting}
	Let $D$ be a Hadamard difference set in the group $G=\langle g \rangle  \times H$, where the order of $g$ is $2^e \geq 4$, and $H$ is an abelian group whose Sylow $2$-subgroup has exponent less than $2^e$. Let $K=\langle g^{2}\rangle \times H$ and $P=\langle g^{2^{e-1}}\rangle\times\{1\}$.
    Suppose that $t$ is an integer such that $t\equiv 1+2^{e-1}\pmod{2^{e}}$ and $t\equiv 1\pmod{\text{exp}(H)}$. If $D$ has $t$ as a multiplier, then $D$ has a Hadamard sub-difference set in $K/P$.
\end{thm}

\begin{proof}
    Suppose that $D$ is fixed by $t$ and write $D=A+PB$ as in Theorem \ref{thm:forbidden}. The equality $|A|=n=\frac{v}{4}$ is forced by the proof of that theorem. Moreover, (\ref{eq:difs}) shows that $AA^{(-1)}$ has coefficients disjoint from $P\setminus\{1\}$. These two facts imply that the elements of $A\subset K$ lie in distinct equivalence classes modulo $P$. Thus, the image of $D$ under the canonical map $\rho:\mathbb{Z}[G]\to \mathbb{Z}[G/P]$ is $\rho(D)=K/P+2hB_0$ for some $B_0\subset K/P$ and $h\in G/P$. By the paragraph preceding this theorem we see that $B_0$ is a difference set of order $\frac{n}{4}$ in a group of order $\frac{v}{4}$. This means that $B_0$ is a Hadamard difference set. 
\end{proof}

\begin{rem}
\label{rem:menon}
 There are examples where $t$ as in Theorem \ref{thm:interesting} is a multiplier. One class of such difference sets are reversible difference sets, \emph{i.e.} those with multiplier $-1$, by \cite[Theorem 4.28]{lander1983symmetric}. There are also examples of non-reversible difference sets; one example is shown in Figure \ref{fig:dif2}. The theorem shows that difference sets in $C_k\times H$ with $|H|$ odd and $k=2^e\leq 8$ cannot have the multiplier.
\end{rem}

\begin{figure}[H]
	\begin{center}
		\small{
			\begin{tabular}{|c|c|c|c|c|c|c|c|c|c|c|c|c|c|c|c|c|}
				\hline
				\ & 0 & 1 & 2 & 3 & 4 & 5 & 6 & 7 & 8 & 9 & 10 & 11 & 12 & 13 & 14 & 15 \\
				\hline
				0& $x$ &$x$&$x$&$x$&$x$&$o$&$x$&$x$&$o$&$x$&$o$&$x$&$o$&$o$&$o$&$x$\\
				1&$x$&$x$&$x$&$o$&$o$&$o$&$x$&$o$&$o$&$x$&$o$&$o$&$x$&$o$&$o$&$o$\\
				2&$x$&$o$&$o$&$o$&$x$&$x$&$o$&$o$&$o$&$o$&$x$&$o$&$o$&$x$&$x$&$o$\\
				3&$x$&$x$&$o$&$o$&$o$&$o$&$o$&$o$&$o$&$x$&$x$&$o$&$x$&$o$&$x$&$o$\\
				\hline
			\end{tabular}
		}
	\end{center}
	\caption{A $(64,28,12)$ Hadamard difference set in $C_{16}\times C_4$ with $t=9$ as a multiplier. Note that column $j+8 \pmod{16}$ is equal to column $j \pmod{16}$ for all odd $j$, and the complement of column $j\pmod{16}$ for all even $j$.}
	\label{fig:dif2}
\end{figure}

\section{Bound on the size of the multiplier group}
\label{sec:M_bound}

Theorem \ref{thm:forbidden}, when relevant, gives a good bound for the size of the numerical multiplier group, since any integer whose power is a forbidden multiplier, is itself a forbidden multiplier.

\begin{thm}
    \label{thm:multiplier_group}
    Let $G=C_m\times H$ be an abelian group where $m=\text{exp}(G)$, and for every prime divisor $p$ of $m$ with $p^e \parallel m$ and $e\geq 2$, the Sylow $p$-subgroup of $H$ has exponent less than $p^e$. Let $D$ be a non-trivial difference set in $G$. Define $c_{m}=1$ if $8\nmid m$ and $c_{m}=2$ otherwise. Then at least one of the following is true.

    \begin{enumerate}
       \item $M(D)$ is isomorphic to a subgroup of $(\mathbb{Z}/c_m\text{rad}(m)\mathbb{Z})^\times$, and in particular
       \begin{align*}
           |M(D)|&\leq c_m \varphi(\text{rad}(m)),
       \end{align*}
          where $\varphi$ is Euler's totient function and $\text{rad}(m)$ is the product of the prime divisors of $m$;
       \item $D$ is a Hadamard difference set containing a multiplier as in Theorem \ref{thm:interesting}.
    \end{enumerate}
\end{thm}

\begin{proof}
    Let $D$ be a difference set not containing a multiplier as in Theorem \ref{thm:interesting}. Identify $M(D)$ with a subgroup of $(\mathbb{Z}/m\mathbb{Z})^\times$ and define $$K=\{t\in(\mathbb{Z}/m\mathbb{Z})^\times:t\equiv 1\pmod{c_m \text{rad}(m)}\}.$$
    Recalling that $(\mathbb{Z}/2^\ell\mathbb{Z})^\times =\langle -1\rangle\langle 5\rangle \cong C_2 \times C_{2^{\ell-2}}$ for $\ell\geq 2$ and $(\mathbb{Z}/p^\ell \mathbb{Z})^\times\cong C_{(p-1)p^{\ell-1}}$ for an odd prime $p$ and all $\ell\geq 1$, we observe that $K$ is cyclic of order $\frac{m}{c_m\text{rad}(m)}$. Let $t\in M(D)\cap K$. If $t$ is non-trivial, then there exists a positive integer $k$ and a prime $p$ with $p^e\parallel m$, $e\geq 2$, such that $t^k\equiv 1+p^{e-1}\pmod{p^e}$ and $t^k\equiv 1\pmod{\text{exp}(C_{m/p^e}\times H)}$. This contradicts Theorem \ref{thm:forbidden}. Thus, $M(D)\cap K=\{1\}$ and $M(D)$ is isomorphic to a subgroup of $(\mathbb{Z}/m\mathbb{Z})^\times /K\cong (\mathbb{Z}/c_m\text{rad}(m)\mathbb{Z})^\times$. The bound on $|M(D)|$ follows.
\end{proof}

From Theorem \ref{thm:multiplier_group} we deduce the next corollary. We will need the fact that $-1$ is never a multiplier in a non-trivial \emph{cyclic} difference set \cite{brualdi1965note} and that no non-trivial cyclic Hadamard difference set, if it exists, has the multiplier of Theorem \ref{thm:interesting}, since $n$ must be odd \cite{turyn1965character}.

\begin{cor}
    \label{cor:multiplier_group_cyclic}
    Let $D$ be a non-trivial cyclic difference set in a group of order $v$ and let $c_v$ be as in Theorem \ref{thm:multiplier_group}. Then $M(D)$ is isomorphic to a subgroup of
	\begin{align*}
		\begin{cases}
			(\mathbb{Z}/\text{rad}(v)\mathbb{Z})^\times /\langle -1\rangle,&\text{ if $v\not \equiv 0\pmod{4}$},\\
            (\mathbb{Z}/c_v\text{rad}(v)\mathbb{Z})^\times,&\text{ if $v\equiv 0\pmod{4}$}.
		\end{cases}
	\end{align*}
In particular, 
    \begin{align*}
    	|M(D)| &\leq \begin{cases}
        \frac{\varphi(\text{rad}(v))}{2},&\text{ if $v\not \equiv 0\pmod{4}$},\\
    		c_v\varphi(\text{rad}(v)),&\text{ if $v\equiv 0\pmod{4}$}.
    	\end{cases}
    \end{align*}
\end{cor}
\begin{proof}
    If $v\equiv 0\pmod{4}$ then the claim follows from Theorem \ref{thm:multiplier_group} and the paragraph preceding this corollary. If $v\not\equiv 0\pmod{4}$, then consider $\Tilde{K}=\langle -1\rangle K$ where $K$ is as in the proof of Theorem \ref{thm:multiplier_group}. Note that $M(D)\cap\Tilde{K}=\{1\}$, as $M(D)\cap K=\{1\}$, $K$ has odd order and $-1\not\in M(D)$. Therefore, $M(D)$ is isomorphic to a subgroup of $(\mathbb{Z}/v\mathbb{Z})^\times/\Tilde{K}\cong (\mathbb{Z}/\text{rad}(v)\mathbb{Z})^\times /\langle -1\rangle$.
\end{proof}

In the cyclic case, Xiang and Chen \cite{xiang1995size} proved the following theorem.

\begin{thm}[\cite{xiang1995size}]
\label{thm:XiangChen}
    Let $D$ be a non-trivial cyclic $(v,k,\lambda)$ difference set. Then the multiplier group $M(D)$ satisfies $|M(D)|\leq k$, except when $D=\{3,6,7,12,14\}$ is the cyclic $(21,5,1)$ difference set associated with a projective plane. 
\end{thm}

We now list a couple of examples where Corollary \ref{cor:multiplier_group_cyclic} gives a smaller bound than Theorem \ref{thm:XiangChen}. Cyclic difference sets arising from the point-hyperplane construction of the projective geometry $\text{PG}(m,q)$ have parameters $$(v,k,\lambda)=\left(\frac{q^{m+1}-1}{q-1},\frac{q^m-1}{q-1},\frac{q^{m-1}-1}{q-1}\right).$$ When $(m,q)=(5,2)$, one obtains a $(63,31,15)$ difference set with multiplier group $$M(D)=\{1,2,4,8,16,32\}.$$ Corollary \ref{cor:multiplier_group_cyclic} gives the upper bound $|M(D)|\leq\frac{\varphi(\text{rad}(v))}{2}= 6$, which is tight, whereas Theorem \ref{thm:XiangChen} gives $|M(D)|\leq k=31$. For another example, when $(m,q)=(4,3)$, one obtains a $(121,40,13)$ cyclic difference set with multiplier group $\{1,3,9,27,81\}$. Corollary \ref{cor:multiplier_group_cyclic} shows that $|M(D)|\leq \frac{\varphi(\text{rad}(v))}{2}=5$, which is tight, whereas Theorem \ref{thm:XiangChen} gives the bound $|M(D)|\leq k=40$.

\section{A multiplier theorem and non-existence of difference sets}\label{sec:mults}

As it turns out, Theorem \ref{thm:forbidden} together with a strong enough multiplier theorem can be used to rule out difference sets. Finding conditions that ensure the existence of multipliers is a problem in its own right. One of the earliest multiplier theorems is presented below. From now on, we use the notation $v^*=\text{exp}(G)$ for the exponent of a group $G$, and $\mathbb{N}[G]$ for the semiring of group ring elements with non-negative coefficients.

\begin{thm}[Menon]\label{thm:menon}
	Let $D$ be a $(v,k,\lambda)$ difference set and let $t$ be an integer relatively prime to $v$. Let $n_1$ be a divisor of $n$ relatively prime to $v$, and suppose that for each prime divisor $p\mid n_1$, there exists an integer $j_p$ such that $t\equiv p^{j_p}\pmod{v^{*}}$. If $n_1>\lambda$, then $t$ is a multiplier of $D$.
\end{thm}

The multiplier conjecture asserts that the condition $n_1>\lambda$ in Theorem \ref{thm:menon} is not needed. A method for weakening this condition is introduced in \cite{mcfarland1970multipliers} and developed further in \cite{leung2014multiplier}. The following simple observation is key behind the multiplier theorems of \cite{mcfarland1970multipliers} and \cite{leung2014multiplier}. 

\begin{lem}
	\label{lem:F_mult}
	Let $D$ be a $(v,k,\lambda)$ difference set and let $t$ be an integer relatively prime to $v$. Let $X=D^{(t)}D^{(-1)}-\lambda G$. Then $|X|=n$, $XX^{(-1)}=n^2$ and 
	$$X\in\mathbb{N}[G] \iff X=ng\text{ for some $g\in G$ } \iff t\in M(D).$$
\end{lem}

Thus, existence of multipliers can be proved by showing that the equation $XX^{(-1)}=n^2$ admits only trivial solutions over $\mathbb{Z}[G]$, \emph{i.e.}, solutions of the form $X=\pm ng$ for some $g\in G$. For small $n$, all solutions can be characterized explicitly. For general $n$, McFarland proved \cite{mcfarland1970multipliers} the following. Let $M'(n)$ be the number defined for $n\leq 4$ by setting
\begin{align}
	\label{eq:mf}
	M'(1)=1,\ M'(2)=2\cdot 7,\ M'(3)=2\cdot 3\cdot 11\cdot 13,\ M'(4)=2\cdot 3\cdot 7\cdot 31.
\end{align}
For $n\geq5$, let $p$ be a prime divisor of $n$ such that $p^e \parallel n$. Then define recursively
$$M'(n)=\text{rad}\left\{ n\cdot M'\left(\frac{n^2}{p^{2e}}\right)\cdot \prod_{j=1}^{\frac{n^2-n}{2}} (p^j-1) \right\}.$$

\begin{thm}[\cite{mcfarland1970multipliers}]
	\label{thm:m_function}
	Let $G$ be a finite abelian group and $n\geq 1$ an integer such that $\gcd(M'(n),|G|)=1$. Then the only solutions to $XX^{(-1)}=n^2$ for $X\in \mathbb{Z}[G]$ are the trivial ones.
\end{thm}

The authors of \cite{leung2014multiplier} show that Theorem \ref{thm:m_function} remains valid when $M'(n)$ is replaced with a number $M(n)$ with much less prime divisors than $M'(n)$. It is defined as follows. Define $M(1,b)=1$ for all integers $b$, and for $n>1$, let
\begin{align*}
M(n,b)&=\text{rad}\left\{n\cdot M\left(\frac{n^2}{p^{2e}},\frac{2n^2}{p^{2e}}-2\right)\cdot\prod_{j=1}^{b}(p^j-1)\right\},
\end{align*}
for any prime divisor $p$ of $n$ with $p^e \parallel n$. Let $$M(n)=f(4n-1)M(n,2n-2),$$ where $f(m)=m$ if $m$ is prime and $f(m)=1$ otherwise. Then Theorem \ref{thm:m_function} remains valid with $M(n)$ in place of $\Tilde{M}(n)$.

This raises the question whether it is possible to go further. By considering a parity condition that the coefficients of a solution to $XX^{(-1)}=n^2$ must satisfy, we show next that it is indeed possible to reduce the number of prime divisors of $M(n)$. Define 
$$S_n=\{1,\dots,n-1\}\cup\{n\leq j\leq 2n-3: \text{ $j$ odd}\}.$$
Put $\tilde{M}(1,b)=1$ for all integers $b$, and for $n>1$, define $\Tilde{M}(n,b)$ recursively via
\begin{align}
\label{eq:mfunction_modified_def}
	\Tilde{M}(n,b)=\text{rad}\left\{n\cdot \Tilde{M}\left(\frac{n^2}{p^{2e}},\frac{2n^2}{p^{2e}}-3\right)\cdot \prod_{j\in S_n,j\leq b}(p^j-1)\right\},
\end{align}
where $p$ is any prime divisor of $n$ such that $p^e \parallel n$. We set $$\Tilde{M}(n)=f(4n-1)\tilde{M}(n,2n-3).$$
Table \ref{tab:M_comparison} compares the number of prime divisors of different $M$-functions for small $n$. Note that the value of an $M$-function depends on the order in which the prime divisors of the argument are selected for the recursion, but Theorem \ref{thm:m_function} holds no matter the order. 

\begin{table}[H]
\label{tab:M_comparison}
    \centering
    \begin{tabular}{|c|c|c|c|}
    \hline
        $n$ &  $\tau(M'(n))$ & $\tau(M(n))$ & $\tau(\Tilde{M}(n))$ \\
        \hline
        5 & 13& 11& 9\\
        6 & 17& 12&10\\
        7 & 33& 18& 13\\
        8 & 31& 14&10\\
        9 & 55& 18&13\\
        10 & 90& 25&18\\
        11 & 128& 28&21\\
        12 & 132& 31& 24\\
        %13 & ? & 46& 35\\
        %14 & ?& 45& 31\\
        %15 & ?& 59&45\\
        \hline
    \end{tabular}
    \caption{Comparison of the number of prime divisors of different versions of McFarland's $M'$-function. Here, $\tau(m)$ is the number of prime divisors of $m$.}
    \label{tab:my_label}
\end{table}

Before giving the proof of Theorem \ref{thm:m_function} with $\tilde{M}$, we record some consequences of the new $M$-function. The next theorem is a variant of \cite[Theorem 3.1]{gordon2016survey}, the difference being that $M$ is replaced by $\tilde{M}$, and we allow the possibility that $\gcd(n_1,v)>1$. The proof is similar to the proof of \cite[Theorem 3.1]{gordon2016survey} so it is omitted. We let $\nu_p$ be the $p$-adic valuation.

\begin{thm}
	\label{thm:multiplier}
	Let $D$ be an abelian $(v,k,\lambda)$ difference set in a group $G$. Let $t$ be an integer relatively prime to $v$. Let $n_1$ be a divisor of $n$ such that for every prime divisor $p\mid n_1$, at least one of the following holds.
	\begin{enumerate}
		\item $p$ is self-conjugate modulo $v^*$;
		\item $t\equiv p^{j_p}\pmod{v^{*}/p^{\nu_p(v^{*})}}$ for some integer $j_p$.
	\end{enumerate}
	If 
	\begin{align*}
		\frac{n_1}{\gcd(n_1,v)}>\lambda \text{ or } \gcd\left(v,\tilde{M}\left(\frac{n \gcd(n_1,v)}{n_1},\left\lfloor\frac{k \gcd(n_1,v)}{n_1}\right\rfloor\right)\right)=1,
	\end{align*}
    then $t$ is a multiplier of $D$.
\end{thm}

Table \ref{tab:difsetss} lists difference sets that are ruled out by Theorems \ref{thm:forbidden} and \ref{thm:multiplier}; we use the database \cite{repo} as a standard reference for open cases of abelian difference sets. For these parameter tuples the condition $\frac{n_1}{\gcd(n_1,v)}>\lambda$ of Theorem \ref{thm:multiplier} with $n_1=n$ suffices. The value of the forbidden multiplier $t$ is shown in the table.

\begin{table}[H]
	\centering
	\begin{tabular}{ l r r }
		$(v,k,\lambda)$ & Group & $t$\\
		\hline 
		$(343,171,85)$ & $C_{7}\times C_{49}$ & $\frac{\text{exp}(G)}{7}+1$\\ 
		$(768,118,18)$ & $C_4\times C_{192}$, $C_8\times C_{96}$ & $\frac{\text{exp}(G)}{2}+1$\\
		$(5859,203,7)$ & $C_{5859}$ & $\frac{\text{exp}(G)}{3}+1$\\
		$(675,337,168)$ & $C_5\times C_{135}$ & $\frac{\text{exp}(G)}{3}+1$\\
		$(783,391,
        195)$ & $C_3\times C_{261}$ & $\frac{\text{exp}(G)}{3}+1$
	\end{tabular}
\caption{Difference sets ruled out by Theorems \ref{thm:forbidden} and \ref{thm:multiplier}.}\label{tab:difsetss}
\end{table}

\subsection{Proofs}

In order to prove Theorem \ref{thm:m_function} with $\tilde{M}$, it will be enough to prove the next theorem, since then Theorem \ref{thm:m_function} follows exactly as Theorem 3.2 in \cite{leung2014multiplier} follows from Theorem 3.3 and Theorem 3.4 in \cite{leung2014multiplier}. Throughout, $G$ denotes a finite abelian group and $t$ a fixed integer relatively prime to $|G|$.

\begin{thm}
    \label{thm:mult_1}
    Let $X\in \mathbb{Z}[G]$ satisfy $X^{(t)}=X$, $|X|=n$ and $XX^{(-1)}=n^2$ for some integer $n\geq 1$. Let $a_0$ be the coefficient of the identity of $X$. Suppose that for every prime divisor $q\mid |G|$ at least one of the following is true:
    \begin{enumerate}
        \item $\text{ord}_q(t)>\max\{n,n-a_0\}$;
        \item $\text{ord}_q(t)\geq n$ and $\text{ord}_q(t)$ is even. 
    \end{enumerate}
    Then $X=n$.
\end{thm}
\begin{rem}\label{rem:non-neg}
    It follows implicitly from the proof of Theorem 3.3 in \cite{leung2014multiplier} that if $X\in \mathbb{Z}[G]$ is a non-trivial solution satisfying the assumptions of Theorem \ref{thm:mult_1}, then $\text{ord}_q(t)\geq n$ for all prime divisors $q\mid |G|$ implies that $a_0\leq 0$.
\end{rem}
We need a couple of lemmas to obtain this theorem.
\begin{lem}
    \label{lem:trivial}
    The length of the orbit $g\mapsto g^t$, where $g\in G\setminus\{1\}$, is divisible by $\text{ord}_q(t)$ for some prime divisor $q\mid |G|$.
\end{lem}
We can now prove Theorem \ref{thm:mult_1} under a weaker assumption.

\begin{lem}\label{lem:trivial2}
    Let $X\in\mathbb{Z}[G]$ be as in Theorem \ref{thm:mult_1}, and suppose that $x=\gcd(\text{ord}_q(t): q\mid |G|)$ satisfies condition 1. or condition 2. of that theorem. Then $X=n$.
\end{lem}
\begin{proof}
    Suppose $X\neq n$ and write $X=a_0+\sum_{i=1}^\ell a_i A_i$, where $A_i$ are disjoint non-identity orbits of $g\mapsto g^t$ and $a_i$ are integers. Then $|X|=n$ and $XX^{(-1)}=n^2$ imply 
    \begin{align}
        a_0+\sum_{i=1}^\ell a_i |A_i|=n, \quad a_0^2+\sum_{i=1}^\ell a_i^2 |A_i|=n^2.
    \end{align}
    By Lemma \ref{lem:trivial2}, we can write $\sum_{i=1}^\ell a_i |A_i|=xy$ and $\sum_{i=1}^\ell a_i^2 |A_i|=x\tilde{y}$ for some integers $y,\tilde{y}$ with $y\equiv \tilde{y} \pmod{2}$. If condition 1. holds, then $0<y=\frac{n-a_0}{x}<1$ or $y=0$, a contradiction. On the other hand, if condition 2. holds then similarly $y\in\{0,1\}$. But the equations $a_0+xy=n$ and $a_0^2+x\tilde{y}=n^2$ and the assumption $x\equiv 0\pmod{2}$ easily imply $\tilde{y}\equiv 0\pmod{2}$ whence $y=0$, a contradiction.
\end{proof}
The general case is obtained by performing a simple induction on the number of prime divisors of $|G|$.
\begin{proof}[Proof of Theorem \ref{thm:mult_1}]
     The proof is by induction on the number $r$ of prime divisors of $|G|$. If $r=0$, there is nothing to prove, so suppose $G=\prod_{i=1}^r G_i$, where $G_i$ is the Sylow $q_i$-subgroup of $G$ and $r\geq 1$. Suppose $X\neq n$ is a solution in $\mathbb{Z}[G]$ satisfying the assumptions of the theorem and suppose that the theorem holds for all abelian groups whose order has less than $r$ prime divisors. By Remark \ref{rem:non-neg} we may suppose that $a_0\leq 0$. Let $\rho_i:\mathbb{Z}[G]\to\mathbb{Z}[\tilde{G_i}]$ be the canonical epimorphism, where $\tilde{G_i}=G_1\times\cdots \times G_{i-1}\times G_{i+1}\times\cdots\times G_r$. Write $X=a_0+X_i+Y_i$, where the support of $X_i$ is contained in $\ker(\rho_i)\setminus\{1\}$ and the support of $Y_i$ is disjoint from $\ker(\rho_i)$. Then the assumption $X^{(t)}=X$ implies that $X_i^{(t)}=X_i$, and thus, by Lemma \ref{lem:trivial},
    \begin{align}
        \rho_i(X)=(a_0+z_i \text{ord}_{q_i}(t))+\rho_i(Y_i)\quad\text{for some $z_i\in \mathbb{Z}$, $\forall i=1,\dots,r$.}
    \end{align}
    Note that $|a_0+z_i\text{ord}_{q_i}(t)|\leq n$ for all $i$. Thus, if there exists an $i$ such that $\text{ord}_{q_i}(t)>n-a_0$, then $z_i=0$ and the coefficient of the identity of $\rho_i(X)$ is $a_0$. In particular, $\rho_i(X)$ satisfies the assumptions of Theorem \ref{thm:mult_1}, so by induction, $a_0=n$, a contradiction. Therefore, $\text{ord}_{q_i}(t)\geq n$ is an even integer for all $i$. But then $\rho_i(X)$ satisfies the assumptions of Theorem \ref{thm:mult_1}, which means that $a_0+z_i\text{ord}_{q_i}(t)=n$ for all $i$. Then the only possibility is $z_i=1$ for all $i$, or $\text{ord}_{q_1}(t)=\cdots=\text{ord}_{q_r}(t)$. In this case, Lemma \ref{lem:trivial2} shows that $X=n$.
\end{proof}

\bibliographystyle{plain}

\bibliography{refs}

\end{document}